\documentclass[12pt]{amsart}
\usepackage[utf8]{inputenc} 
\usepackage[T1]{fontenc}
\usepackage{color}
\usepackage{comment}
\usepackage{times}
\usepackage[hidelinks]{hyperref}
\usepackage{url}  
\usepackage{enumerate,latexsym}
\usepackage{graphicx}
\usepackage{subcaption}
\usepackage{array}
\usepackage{makecell}
\usepackage{mathtools}

\newcolumntype{M}[1]{>{\centering\arraybackslash}m{#1}}
\usepackage{longtable}

\usepackage{fullpage}

\usepackage{amsmath,amsthm,amsfonts,amssymb}
\usepackage{graphicx}
\usepackage{dsfont}
\usepackage{tkz-berge}
\usepackage{tikz-cd}
\usepackage{float}
\usetikzlibrary{knots, decorations.markings, positioning, arrows.meta, decorations.pathreplacing}
\usepackage{tkz-berge}

\makeatletter
\def\namedlabel#1#2{\begingroup
 #2%
 \def\@currentlabel{#2}%
 \phantomsection\label{#1}\endgroup
}
\makeatother

\usepackage[lite]{amsrefs}

\renewcommand{\PrintDOI}[1]{\href{http://dx.doi.org/\detokenize{#1}}{doi: \detokenize{#1}}%
	\IfEmptyBibField{pages}{, (to appear in print)}{}}

\theoremstyle{plain}
\newtheorem*{theorem*}{Theorem}
\newtheorem*{thmex*}{Theorem~\ref{example}}
\newtheorem*{thmasymp*}{Theorem~\ref{thmAsymp}}
\newtheorem{theorem}{Theorem}[section]

\newtheorem{corollary}[theorem]{Corollary}

\newtheorem{remark}[theorem]{Remark}

\newtheorem{example}[theorem]{Example}

\theoremstyle{definition}
\newtheorem{definition}[theorem]{Definition}

\newcommand{\ben}{\begin{enumerate}}
\newcommand{\een}{\end{enumerate}}

\newcommand{\ed}{\end{document}}

\definecolor{rrr}{rgb}{.9,0,.1}

\definecolor{rr}{rgb}{.8,0,.3}

\usepackage{pdfsync}
\graphicspath{ {./images/} }

\title[Subquandle Quiver Decategorification]{Polynomial Link Invariants via Decategorification of Subquandle Quiver Filtrations}

\author[J. Ceniceros]{Jose Ceniceros}
\address{Mathematics and Statistics Department, Hamilton College, Clinton, NY, USA}
\email{jcenicer@hamilton.edu}

\author[B. Chidley]{Benjamin Chidley}
\address{Mathematics and Statistics Department, Hamilton College, Clinton, NY, USA}
\email{bchidley@hamilton.edu}

\author[J.Vaughn]{Jackson Vaughn}
\address{Mathematics and Statistics Department, Hamilton College, Clinton, NY, USA}
\email{jvaughn@hamilton.edu}

\usetikzlibrary{external}
\begin{document}

\begin{abstract}
In this paper, we introduce the notion of subquandle filtrations for finite quandles as a foundation to construct algebraic and combinatorial link invariants. A nested sequence of subquandles naturally induces a filtration of the associated quandle coloring quiver by subquivers. By decategorifying these subquivers, we define two new polynomial-valued invariants of links: the disrespectful polynomial and the graded disrespectful polynomial, which measure how a chosen set of quandle endomorphisms respects and disrespects subquandle structure. Finally, we illustrate the distinguishing power of these new tools by comparing them to both classical quandle counting invariants and the quiver in-degree polynomial.
\end{abstract}

\maketitle
\section{Introduction}

\emph{Quandles} are algebraic structures whose defining axioms encode the Reidemeister moves of classical knot theory. First introduced independently by Joyce \cite{Joyce} and Matveev \cite{Matveev}, the theory associates to each link $L$ a fundamental quandle, $Q(L)$. Joyce and Matveev demonstrated that $Q(L)$ is a complete invariant of links up to mirror image and orientation reversal. Despite its theoretical power, the fundamental quandle is a difficult object to compare directly. Consequently, a major focus in the field has been the development of computable invariants derived from $Q(L)$. A standard approach is the quandle counting invariant, which computes the cardinality of the set of homomorphisms from $Q(L)$ to a finite target quandle $X$. Subsequent work has focused on generalizing and enhancing this counting invariant to capture even more topological information \cites{CJKLS, FJK, N5}.

A specific enhancement that plays a central role in this article was introduced by Cho and Nelson \cite{CN}. By considering subsets of the set of endomorphisms $\operatorname{End}(X)$ of a quandle $X$, they categorified the quandle counting invariant of links. Specifically, they defined a quiver-valued invariant of links called the \emph{quandle coloring quiver}. From this quiver, Cho and Nelson derived families of polynomial invariants via decategorification, which they called the \emph{in-degree polynomial}. Since the introduction of the quandle coloring quiver, several researchers have studied properties and generalizations of the quandle coloring quiver; see \cites{BaCa, CCN, CN2, ZhLi}.

In this article, we introduce the concept of \emph{subquandle filtrations} for finite quandles as a foundation to construct enhancements to the quandle counting invariant. A subquandle filtration naturally induces both a sequence of classical counting invariants and a filtration of the quandle coloring quiver. By leveraging these subquivers, we derive two new and effective families of polynomial invariants. At each filtration stage, we introduce the \emph{disrespectful polynomial}, which characterizes the behavior of link quandle colorings under a chosen set of endomorphisms by partitioning endomorphisms based on their preservation or violation of the subquandle structure. Building upon this, we define the \emph{graded disrespectful polynomial}, a jazzed-up disrespectful polynomial that explicitly tracks coloring depths across the filtration levels and a new measure we call the defect weights. This approach allows us to capture structural distinctions between links that classical counting methods and the ungraded disrespectful polynomial miss. Additionally, while the in-degree polynomial introduced by Cho and Nelson \cite{CN} focuses on the incoming arrows at the vertices of the quiver, our graded disrespectful polynomial explicitly incorporates the layer structure induced by the filtration. Thus, our graded disrespectful polynomial can distinguish knots and links with the same in-degree polynomial, as we demonstrate in Section~\ref{Examples}.

The paper is organized as follows. Section~\ref{QuandlesandQuivers} reviews the basics of links, quandles, and quandle quivers. Section~\ref{Subquandlefiltration} introduces a subquandle filtration that induces a filtration on the quandle quiver, from which we derive a new package of invariants. Finally, Section~\ref{Examples} provides computations demonstrating that these invariants enhance existing ones or provide distinct topological information.

\section{Quandles and Quivers}\label{QuandlesandQuivers}

In this section, we give a brief overview of the basic definitions and results of quandles and quivers; for more details, see \cites{EN,Joyce,Matveev}. We begin with the following definition.

\begin{definition}
A set $X$ together with a binary operation $\triangleright$ is a \emph{quandle} if it satisfies the following:
    \begin{enumerate}
        \item for all $x \in X$, $x \triangleright x = x;$
        \item for all $y \in X$, the map $\beta_y : X \to X$ defined by $\beta_y(x) = x \triangleright y$ is a bijection;
        \item for all $x, y, z \in X$, $(x \triangleright y) \triangleright z = (x \triangleright z) \triangleright (y \triangleright z)$.
    \end{enumerate}
It is also common to define condition (2) by explicitly introducing a right inverse operation $\triangleright^{-1}$ such that $(x \triangleright y) \triangleright^{-1} y = x$ and $(x \triangleright^{-1} y) \triangleright y = x$ for all $x, y \in X$. In this article, we will use the notation of the right inverse operation. We also note that $(X, \triangleright^{-1})$ is also a quandle.
\end{definition}

\begin{example}
Below are some common examples of quandles:
    \begin{enumerate}
        \item The set $\mathbb{Z}_n$ forms a quandle $(\mathbb{Z}_n, \triangleright)$, called the dihedral quandle, where $x \triangleright y \equiv 2y - x \pmod{n}$ for all $x, y \in \mathbb{Z}_n$.
        \item All groups $G$ form a quandle $(G, \triangleright)$, called the conjugation quandle, where $x \triangleright y = y^{-1}xy$ for all $x, y \in G$.
        \item There also exist trivial quandles $(X, \triangleright)$ such that $x \triangleright y = x$ for all $x, y \in X$. Note that trivial quandles need not be of order 1; they can be of any order.
    \end{enumerate}
\end{example}

\begin{definition}
Let $(X,\triangleright)$ be a quandle. If a subset $S \subseteq X$ is closed under the operations $\triangleright, \triangleright^{-1}$, then $(S, \triangleright)$ is a \emph{subquandle} of $(X,\triangleright)$.
\end{definition}

In this article, we focus on finite quandles. We can take advantage of the fact that a finite quandle's operation can be fully defined using an operation table. 

\begin{example}
    The dihedral quandle of order $4$, denoted $(\mathbb{Z}_4, \triangleright)$, has the operation $x \triangleright y \equiv 2y - x \pmod 4$, which can be fully represented by the following operation table:
    \[
    \begin{array}{c|cccc}
    \triangleright & 0 & 1 & 2 & 3 \\
    \hline 
    0 & 0 & 2 & 0 & 2 \\
    1 & 3 & 1 & 3 & 1 \\
    2 & 2 & 0 & 2 & 0 \\ 
    3 & 1 & 3 & 1 & 3 \\
    \end{array}.
    \]
\end{example}

Let $(X, \triangleright_X)$ and $(Y, \triangleright_Y)$ be quandles. The map $f: X \to Y$ is a \emph{quandle homomorphism} if $f(x \triangleright_X y) = f(x) \triangleright_Y f(y)$ for all $x, y \in X$. We will denote the set of all homomorphisms from $X$ to $Y$ by $\text{Hom}(X, Y)$. Furthermore,  if $f$ is a bijection, then $f$  is a \emph{quandle isomorphism}. Lastly, the map $g:X \rightarrow X$ is a \emph{quandle endomorphism} if $g(x \triangleright_X y) = g(x) \triangleright_X g(y)$ for all $x, y \in X$. We will denote the set of all endomorphisms of $X$ by $\text{End}(X)$. Let $X = \{x_1, \dots, x_n\}$ be a finite quandle. Moving forward, we will represent an endomorphism $\phi \in \operatorname{End}(X)$ by the $n$-tuple $(\phi(x_1), \phi(x_2), \dots, \phi(x_n))$. Additionally, we will denote a quandle by simply $X$ instead of $(X, \triangleright)$ whenever there is no ambiguity or need for $\triangleright$.

There is a fundamental connection between links and quandles; indeed, the quandle axioms are directly motivated by the Reidemeister moves. Using the coloring rule illustrated in Figure~\ref{fig:ColoringRule}, it is a standard exercise to show that the three quandle axioms correspond precisely to these moves. Furthermore, it suffices to show that the axioms respect a minimal generating set of oriented Reidemeister moves, such as the one introduced by Polyak \cite{P} and included in Figure~\ref{fig:ReidemeisterMoves}.

\begin{figure}[ht]
\includegraphics{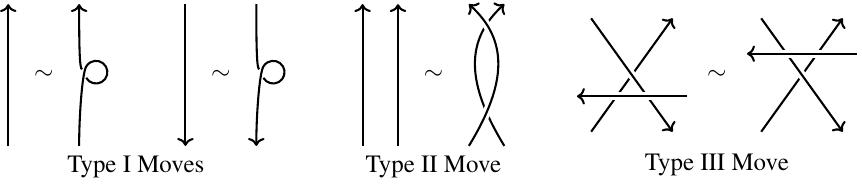}
\caption{A generating set for the oriented Reidemeister moves.}
\label{fig:ReidemeisterMoves}
\end{figure}
Additionally, if we let $L$ be an oriented link with an oriented diagram $D$. Suppose $D$ has $n$ crossings and $n$ arcs. We assign a label to each arc of $D$, and at each crossing, we use the coloring in Figure~\ref{fig:ColoringRule}.

\begin{figure}[ht]
\includegraphics{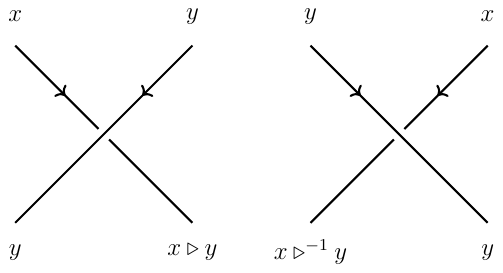}
\caption{Coloring rule at a positive and a negative crossing.}
    \label{fig:ColoringRule}
\end{figure}
\noindent The \emph{fundamental quandle of $L$}, denoted $Q(L)$, is defined as the quandle generated by the arc labels of $D$, modulo the relations given by the crossings. Because the quandle axioms perfectly correspond to the Reidemeister moves, the isomorphism class of the fundamental quandle is independent of the choice of diagram $D$. Therefore, the fundamental quandle is a well-defined link invariant.

Although it is known that the fundamental quandle of a knot is a powerful link invariant, it is difficult to compare two given fundamental quandles directly. Therefore, one way of extracting an effective and computable invariant is to count the homomorphisms from the fundamental quandle of a link to a target finite quandle $(X,\ast)$. The \emph{quandle counting invariant} is defined by
\[ \#Col_X(L) = \vert \text{Hom}(Q(L), X) \vert.\]
Furthermore, we will refer to the elements of $\operatorname{Hom}(Q(L),X)$ as \emph{$X$-colorings} of the link $L$.

We note that while the counting invariant is effective and easily computable, the individual identities of the homomorphisms and their underlying relationships are lost when only the cardinality of $\operatorname{Hom}(Q(L), X)$ for a finite quandle $X$ is considered. In order to address this, significant effort has been devoted to defining enhancements of the quandle counting invariant by introducing additional algebraic and combinatorial structures on $\text{Hom}(Q(L), X)$. One such enhancement that has gained considerable attention is the quiver enhancement introduced by Cho and Nelson \cite{CN}.
\begin{definition}
Let $X$ be a finite quandle, $L$ be a link, and $S \subset \text{End}(X)$. Then, the \emph{$X$-coloring quiver}, denoted $\mathcal{Q}_X^S(K)$, is a directed graph where the vertices are elements $f \in \text{Hom}(\mathcal{Q}(K), X)$ and there exists an edge from $f$ to $g$ if $\phi \circ f = g$ for some $\phi \in S$. In the case when $S = \operatorname{End}(X)$, the $X$-coloring quiver is called the \emph{full quandle coloring quiver} of $L$ with respect to $X$, denoted $\mathcal{Q}_X(L)$. Furthermore, in the case when $S=\{ \phi\}$ is a singleton, we will denote it by $\mathcal{Q}_X^\phi(L)$.
\end{definition}

We will follow the convention set up by Cho and Nelson \cite{CN}: we will use the same notation for a vertex in $\mathcal{Q}_X^S(L)$ and for an $X$-coloring of $L$, since vertices of $\mathcal{Q}_X^S(L)$ correspond precisely with the $X$-colorings of $L$. We will use $\textup{deg}^+(f)$ to denote the in-degree of the vertex $f$ in the quiver $\mathcal{Q}_X^S(L)$. Furthermore, Cho and Nelson derived a polynomial invariant from the $X$-coloring quiver via decategorification. We include the definition of this polynomial for completeness and will be used in Section~\ref{Examples}. 

\begin{definition}
Let $X$ be a finite quandle, $S \subset \operatorname{End}(X)$, and $L$ be an oriented link. Let $\mathcal{Q}_X^S (L)$ be the associated $X$-coloring quiver with vertex set $V(\mathcal{Q}_X^S(L))$. Then the \emph{in-degree quiver polynomial} of $L$ with respect to $X$ is 
\[\Phi_X^{\textup{deg}^+,S}(L) = \sum_{f \in V(\mathcal{Q}_X^S(L))}u^{\textup{deg}^+(f)}.\]    
If $S= \{ \phi \}$, then $\Phi_X^{\textup{deg}^+,S}(L)$ is denoted by $\Phi_X^{\textup{deg}^+,\phi}(L)$, and if $S = \operatorname{End}(X)$, then $\Phi_X^{\textup{deg}^+,S}(L)$ is denoted by $\Phi_X^{\textup{deg}^+}(L)$.
\end{definition}

\section{Subquandle Quiver Filtration} \label{Subquandlefiltration}

In \cite{KH}, the author introduced a quiver filtration based on a nested sequence of subsets of biquandle endomorphisms and proved that the induced filtration of biquandle coloring quivers is a link invariant. 

In this section, we introduce a novel quiver filtration induced by filtering the quandle itself by a nested sequence of subquandles rather than filtering the endomorphism set. We show that the sequence of subquandle counting invariants induced by this filtration is a link invariant, as is the induced quiver filtration. Lastly, we define two polynomial invariants of links via decategorification.

We will start by fixing the notation that will be used throughout this section. Let $L$ be an oriented link with diagram $D$ and fundamental quandle $Q(L)$. Let $X$ be a finite quandle and let $S \subseteq \operatorname{End}(X)$.

\begin{definition}
A \emph{subquandle filtration} of length $m$ is a sequence $U_\ast \coloneqq (U_i)_{i=0}^m$ of subquandles of $X$ satisfying
\[ U_0 \subseteq U_1 \subseteq \dots \subseteq U_m \subseteq X. \]
\end{definition}

For each $i \in \{0, \dots, m\}$, the subquandle $(U_i, \triangleright)$ of $(X,\triangleright)$ is a quandle. Since the fundamental link quandle $Q(L)$ is a link invariant up to isomorphism, the hom-set $\operatorname{Hom}(Q(L), U_i)$ is a link invariant up to bijection. It directly follows that $\#Col_{U_i}(L) = \vert \operatorname{Hom}(Q(L), U_i) \vert$ is an invariant of links and the sequence induced by the subquandle filtration, $U_\ast$,\[  \#Col_{U_\ast}(L) = (\#Col_{U_0}(L),\#Col_{U_1}(L), \dots, \#Col_{U_m}(L))\] is also a link invariant.

Additionally, for each subquandle $U_i \subseteq X$, we obtain a $U_i$-coloring quiver, denoted by $\mathcal{Q}_{U_i}^S(L)$. The vertices of this quiver are the $U_i$-colorings of $L$. We define the directed edges of $\mathcal{Q}_{U_i}^S(L)$ by restricting the edges of the $X$-coloring quiver $\mathcal{Q}_X^S(L)$: a directed edge labeled by $\alpha \in S$ exists from vertex $f$ to vertex $\alpha \circ f$ if and only if both $f$ and $\alpha \circ f$ are valid $U_i$-colorings of $L$. When this holds, we call the endomorphism $\alpha$ a \emph{respectful endomorphism at $f$}. Otherwise, we call it a \emph{disrespectful endomorphism at $f$}. 

For any endomorphism $\alpha \in S$ and any $U_i$-coloring $f \in \operatorname{Hom}(Q(L), U_i)$, the directed edge from $f$ to $\alpha \circ f$ in the quiver $\mathcal{Q}_{U_i}^S(L)$ is strictly determined by the action of $\alpha$ on $\operatorname{Hom}(Q(L), U_i)$. Therefore, each quiver $\mathcal{Q}_{U_i}^S(L)$ is completely determined up to quiver isomorphism by the fundamental quandle $Q(L)$, the subquandle $U_i$, and the set $S$. Since the subquandle filtration $U_\ast$ and the set $S\subseteq \operatorname{End}(X)$ are fixed, each quiver $\mathcal{Q}_{U_i}^S(L)$ is completely determined up to graph isomorphism by the fundamental quandle $Q(L)$ and hence is a link invariant.
Furthermore, since $U_i \subseteq U_{i+1}$, any valid $U_i$-coloring of $L$ is naturally a $U_{i+1}$-coloring. This implies that $V(\mathcal{Q}_{U_i}^S(L)) \subseteq V(\mathcal{Q}_{U_{i+1}}^S(L))$. It follows that if an edge exists between $f$ and $\alpha \circ f$ in $\mathcal{Q}_{U_i}^S(L)$, both endpoints are valid $U_{i+1}$-colorings, meaning the exact same edge exists in $\mathcal{Q}_{U_{i+1}}^S(L)$. Therefore, $\mathcal{Q}_{U_i}^S(L)$ is a subquiver of $\mathcal{Q}_{U_{i+1}}^S(L)$, which we denote by $\mathcal{Q}_{U_i}^S(L) \subseteq \mathcal{Q}_{U_{i+1}}^S(L)$.

\begin{definition}
Let $L$ be an oriented link, $X$ a finite quandle, $S \subset \operatorname{End}(X)$, and $U_\ast$ a subquandle filtration of $X$. The \emph{subquandle quiver filtration} of $L$ is a sequence $\mathcal{Q}_{U_\ast}^S(L) \coloneqq (\mathcal{Q}_{U_i}^S)_{i=0}^m$ of subquivers of $\mathcal{Q}_{X}^S(L)$ satisfying
\[ \mathcal{Q}_{U_0}^S(L) \subseteq \mathcal{Q}_{U_1}^S(L) \subseteq \dots \subseteq \mathcal{Q}_{U_m}^S(L). \]
If $S = \operatorname{End}(X)$, we omit the superscript and write $\mathcal{Q}_{U_\ast}(L)$.
\end{definition}

 As noted in \cite{KH}, while a filtration traditionally terminates at the ambient object, we relax this condition and define a filtration with the flexibility to terminate early. This flexibility allows us to build a nested subquandle sequence from a base subquandle and terminating the sequence at any intermediate level. Furthermore, because our aim is to distinguish links, these intermediate levels often capture enough information to distinguish knots and links without requiring the full ambient quandle. To prove the following results, we use a similar argument to that of \cite[Theorem 3.2.1]{KH}.

\begin{theorem}
The subquandle quiver filtration $\mathcal{Q}_{U_\ast}^S(L)$ is a link invariant.
\end{theorem}

\begin{proof}
If $L$ and $L'$ are isotopic links, there exists a quandle isomorphism $Q(L) \cong Q(L')$. For each filtration level $i$, this induces a bijection between $\operatorname{Hom}(Q(L), U_i)$ and $\operatorname{Hom}(Q(L'), U_{i})$  that preserves the quiver structure, yielding a quiver isomorphism $\phi_i \colon \mathcal{Q}_{U_i}^S(L) \rightarrow \mathcal{Q}_{U_i}^S(L')$. Furthermore, because the inclusions $\mathcal{Q}_{U_i}^S \subseteq \mathcal{Q}_{U_{i+1}}^S$ are induced by the subquandle inclusions $U_i \subseteq U_{i+1}$ and each $\phi_i$ is induced by the same quandle isomorphism, the sequence of isomorphisms $\phi_i$ naturally commutes with these inclusion maps. That is, the restriction of $\phi_{i+1}$ to the subquiver $\mathcal{Q}_{U_i}^S(L)$ is exactly $\phi_i$. Since both the quivers at every level and the nested inclusions between them are preserved, the entire subquandle quiver filtration $\mathcal{Q}_{U_\ast}^S(L)$ is a link invariant.
\end{proof}

\begin{corollary}\label{invariant}
Any invariant of directed graph filtrations evaluated on $\mathcal{Q}_{U_\ast}^S(L)$, or any directed graph invariant applied at each level $\mathcal{Q}_{U_i}^S(L)$, is an oriented link invariant.
\end{corollary}

For the filtration level $i$, every vertex $f \in V(\mathcal{Q}_{U_i}^S(L))$ has up to $\vert S \vert$ potential edges coming out of the vertex. Based on whether the endomorphism maps the $U_i$-coloring to another valid $U_i$-coloring, $S$ partitions at the vertex $f$ into two disjoint subsets.
\begin{definition}
    Let $f \in V(\mathcal{Q}_{U_i}^S (L))$. The set of endomorphisms $S \subseteq \operatorname{End}(X)$ partitions into a \emph{respectful set} and a \emph{disrespectful set} relative to the vertex $f$, defined as:

    \[ S_{res}(f,U_i) = \{ \phi \in S \, \vert \, \phi \circ f \text{ is a valid $U_i$-coloring}\},\]
    \[ S_{dis}(f, U_i)  = \{ \phi \in S \, \vert \, \phi \circ f \text{ is not a valid $U_i$-coloring}\}.\]

\end{definition}

In the construction above of each $\mathcal{Q}_{U_i}^S$, we restrict our attention to respectful endomorphisms. Thus, we have that the out-degree of $f$ in the quiver $\mathcal{Q}_{U_i}^S(L)$ is exactly $\vert S_{res}(f, U_i) \vert$. Note that since for each vertex $f \in V(\mathcal{Q}_{U_i}^S(L))$ we have that $\vert S_{res}(f,U_i) \vert + \vert S_{dis}(f,U_i) \vert = \vert S \vert$ is constant, we can keep track of the disrespectful endomorphisms with a single variable polynomial.

\begin{definition}
Let $\mathcal{Q}_{U_\ast}^S(L)$ be a subquandle quiver filtration of $L$. The \emph{disrespectful polynomial} of $\mathcal{Q}_{U_i}^S(L)$ at filtration level $i$ is defined as
    \[ \mathbf{dis}_{U_i}^S(L)(t) = \sum_{f \in V(\mathcal{Q}_{U_i}^S(L))} t^{\vert S_{dis}(f,U_i)\vert}.\]
Thus, for the subquandle quiver filtration $U_\ast$, we obtain a sequence of disrespectful polynomials: 
\[\mathbf{dis}_{U_\ast}^S(L)(t) = \left( \mathbf{dis}_{U_1}^S(L)(t), \mathbf{dis}_{U_2}^S(L)(t), \dots, \mathbf{dis}_{U_m}^S(L)(t)\right).\]
\end{definition}

\begin{corollary}
For each $i$, the disrespectful polynomial $\mathbf{dis}_{U_i}^{S}(L)(t)$, as well as the sequence of disrespectful polynomials $\mathbf{dis}_{U_\ast}^S(L)$, are link invariants.
\end{corollary}

\begin{proof}
    This follows immediately from Corollary~\ref{invariant}.
\end{proof}

We will further extract more information by not just using the subquandle structure, but its specific position within the filtration.  

\begin{definition}
Let $f \in \operatorname{Hom}(Q(L), X)$ be an $X$-coloring of $L$. For an endomorphism $\phi \in \operatorname{End}(X)$,  the \emph{target index} of $\phi$ relative to $f$, denoted $\operatorname{idx}(\phi, f)$, is defined as
\[
\operatorname{idx}(\phi, f) = \begin{cases} 
\min \{ j \in \{0, \dots, m\} \mid \operatorname{Im}(\phi \circ f) \subseteq U_j \} & \text{if } \operatorname{Im}(\phi \circ f) \subseteq U_m, \\ 
\infty & \text{otherwise.} 
\end{cases}
\]
\end{definition}

Note that the target index of $\phi$ relative to $f$ measures the minimal filtration level $j$ such that $\operatorname{Im}(\phi \circ f) \subseteq U_j$. Additionally, to simplify notation, we introduce the following two index sets: 
\[ M = \{0, 1, \dots , m \} \quad \text{and} \quad \bar{M} = M \cup \{\infty\}. \]

\begin{definition}\label{Disrespectfullayers}
Given a subquandle filtration $U_\ast$ of $X$, we partition the quandle $X$ into layers $B_j \subseteq X$ for each $j \in \bar{M}$:
\[
B_j = 
\begin{cases} 
U_0 & \text{if } j=0, \\
U_j \setminus U_{j-1} & \text{if } j \in \{1, \dots, m\}, \\ 
X \setminus U_m & \text{if } j = \infty. 
\end{cases}
\]
Relative to a fixed filtration level $i \in M$, we say $B_j$ is a \emph{respectful layer} if $j \leq i$, and a \emph{disrespectful layer} if $j > i$ or $j = \infty$.
\end{definition}

\begin{definition}
Let $f \in \operatorname{Hom}(Q(L), X)$, and let $\phi \in S \subseteq \operatorname{End}(X)$ with $\operatorname{idx}(\phi, f) = j$. The \emph{defect weight} of $\phi$ at $f$, denoted $\delta(\phi, f) \in \mathbb{N}$, is defined as 
\[
\delta(\phi, f) = \left\vert{} \operatorname{Im}(\phi \circ f) \cap B_j \right\vert{}.
\]
\end{definition}

We note that because $\operatorname{idx}(\phi,f) = j$, $U_j$ is by definition the smallest subquandle in the filtration containing $\operatorname{Im}(\phi \circ f)$. Thus, the intersection $\operatorname{Im}(\phi \circ f) \cap B_j$ must be nonempty, so we have $\delta(\phi,f) \geq 1$.

\begin{definition}\label{GradedPoly}
Let $\mathcal{Q}_{U_i}^S(L)$ be the induced subquandle quiver, and let $V(\mathcal{Q}_{U_i}^S(L))$ denote its set of vertices (corresponding to the valid $U_i$-colorings $f$ of $L$). For each vertex $f$, we partition the set $S \subseteq \operatorname{End}(X)$ into subsets:
\[
S_{j, k}(f) = \{ \phi \in S \mid \operatorname{idx}(\phi, f) = j \text{ and } \delta(\phi, f) = k \}.
\]

The \emph{graded disrespectful polynomial} of $L$ relative to $U_i$ and $S$ is a multivariable polynomial defined by 
\[ 
G\mathbf{dis}_{U_i}^S(L) = \sum_{f \in V(\mathcal{Q}_{U_i}^S(L))} \left( \prod_{(j,k) \in \bar{M} \times \mathbb{N}} t_{j,k}^{\vert S_{j,k}(f)\vert} \right).
\]
Thus, for the full subquandle filtration $U_\ast$, we obtain a sequence of graded disrespectful polynomials:
\[ 
G\mathbf{dis}_{U_\ast}^S(L) = ( G\mathbf{dis}_{U_0}^S(L), \dots, G\mathbf{dis}_{U_m}^S(L)).
\]
\end{definition}

\begin{corollary}
For each $i$, the graded disrespectful polynomial $G\mathbf{dis}_{U_i}^{S}(L)(t)$, as well as the sequence of disrespectful polynomials $G\mathbf{dis}_{U_\ast}^S(L)$, are link invariants.
\end{corollary}

\begin{proof}
This follows immediately from Corollary~\ref{invariant}.
\end{proof}

\begin{remark}\label{gradedtoungraded}
Given the graded disrespectful polynomial at level $i$, we can recover the disrespectful polynomial by applying a simple variable substitution that collapses all defect weights and maps respectful layers to $1$. Specifically, we evaluate $G\mathbf{dis}_{U_i}^S(L)$ under the substitution:
\[
  t_{j,k} \mapsto  
  \begin{cases}
  t & \text{if } j > i \text{ or } j = \infty, \\
  1 & \text{if } j \leq i.
  \end{cases}
\]
\end{remark}

\begin{example}\label{Ex:TrefoilQuiver}
Let $X$ be the quandle defined by the operation table
\[
\begin{array}{c|cccc} 
     \ast & 0 & 1 & 2 & 3 \\
     \hline
     0 & 0 & 2 & 1 & 0  \\
     1 & 2 & 1 & 0 & 1  \\
     2 & 1 & 0 & 2 & 2  \\
     3 & 3 & 3 & 3 & 3  \\
\end{array}.
\]
Let $S = \{(0, 0, 0, 0), (1, 1, 1, 1), (2, 2, 2, 2), (3, 3, 3, 3), (0, 1, 2, 3)\} \subseteq \operatorname{End}(X)$. Let $D$ be an oriented diagram of $3_1$, see Figure ~\ref{fig:trefoil}.

\begin{figure}[ht]
    \includegraphics{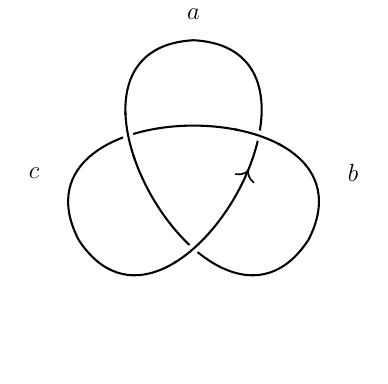}
    \caption{An oriented diagram $D$ of $3_1$.}
    \label{fig:trefoil}
\end{figure}
From the diagram $D$ we obtain the following presentation of the fundamental quandle of $3_1$,
\[
Q(3_1) \cong \langle a, b, c \mid b \triangleright^{-1} a = c, a \triangleright^{-1} c = b, c \triangleright^{-1} b = a \rangle.
\]
Using this presentation of the fundamental quandle of $3_1$, we can compute the $X$-colorings of $3_1$. We represent the elements of  $\operatorname{Hom}(Q(3_1),X)$ with 3-tuples $(f(a), f(b), f(c))$. Thus,
\[
\operatorname{Hom}(Q(3_1), X) = \left\{
\begin{aligned}
&(0, 0, 0), (0, 1, 2), (0, 2, 1), (1, 0, 2), (1, 1, 1), \\
&(1, 2, 0), (2, 0, 1), (2, 1, 0), (2, 2, 2), (3, 3, 3)
\end{aligned}
\right\}.
\]
Hence the quandle counting invariant is $\#\operatorname{Col}_X(3_1)= \vert \operatorname{Hom}(Q(3_1),X)\vert = 10$. 

Let $U_* = (U_i)_{i = 0, 1, 2}$ be a subquandle filtration of $X$, where 
    \begin{align*}
    U_0 &= \{0\},  \\ U_1 &= U_0 \cup \{1,2\}, \\ U_2 &= U_1 \cup \{3\}.
    \end{align*}
We will compute the target index, the defect weight, and the corresponding contribution of $f = (0,1,2), \ \phi = (3,3,3,3)$ to the graded disrespectful polynomial. Note that since $U_0$ only considers colorings that contain $0$, $f$ can only be considered a valid $U_1$ and $U_2$-coloring. For this computation suppose that the filtration level is $i=1$. Observe that $\phi \circ f = (3,3,3)$, so by the first quandle axiom, $\operatorname{Im}(\phi \circ\ f) = \{3\}$. Since $\operatorname{Im}(\phi \circ\ f) \nsubseteq U_0, U_1$ and $\operatorname{Im}(\phi \circ\ f) \subseteq U_2$, $\operatorname{idx}(\phi, f) = 2$. Furthermore, since $B_2 = U_2 \setminus U_1 = \{3\}$, the defect weight of $\phi$ at $f$ is $\delta(\phi, f)= |\operatorname{Im}(\phi \circ f) \cap B_2| = 1$. It follows that $\phi \in S_{2,1}$ relative to the vertex $f$, meaning that $\phi$ at the vertex $f$ contributes the term $t_{2,1}^1$ to the graded disrespectful polynomial $G\mathbf{dis}_{U_1}^S(3_1)$.

Repeating this process for all endomorphisms $\phi \in S \subset \operatorname{End}(X)$ at $f$ gives us $t_{0,1}^{1}t_{1,1}^{2}t_{2,1}^{1}t_{1,2}^{1}$. Repeating this for all endomorphisms $\phi \in S$ and all $f \in \operatorname{Hom}(Q(3_1), U_i)$ for each filtration level $i \in \{0,1,2\}$ gives us the following sequence of graded disrespectful polynomials, see Table~\ref{tab:Gdistabletref}.

\begin{table}[ht]
\centering 
\caption{Sequence of graded disrespectful polynomials for knot $3_1$.}
\label{tab:Gdistabletref}
\renewcommand{\arraystretch}{1.8}
\begin{tabular}{c|l}
Filtration Level & Graded Disrespectful Polynomial, $G\mathbf{dis}_{U_i}^S(3_1)$  \\ \hline
0            
&
$1t_{0,1}^{2}t_{1,1}^{2}t_{2,1}^{1}$                               \\ \hline
1          & $1t_{0,1}^{2}t_{1,1}^{2}t_{2,1}^{1} + 6t_{0,1}^{1}t_{1,1}^{2}t_{2,1}^{1}t_{1,2}^{1} + 2t_{0,1}^{1}t_{1,1}^{3}t_{2,1}^{1}$ 
                  \\ \hline
2          & $1t_{0,1}^{2}t_{1,1}^{2}t_{2,1}^{1} + 6t_{0,1}^{1}t_{1,1}^{2}t_{2,1}^{1}t_{1,2}^{1} + 2t_{0,1}^{1}t_{1,1}^{3}t_{2,1}^{1} + 1t_{0,1}^{1}t_{1,1}^{2}t_{2,1}^{2}$

\end{tabular}
\renewcommand{\arraystretch}{1}
\end{table}

Furthermore we collect the disrespectful polynomials derived from the graded polynomials using the method described in Remark~\ref{gradedtoungraded} in Table \ref{tab:distabletref}.

\begin{table}[ht]
\centering 
\caption{Sequence of disrespectful polynomials for knot $3_1$.}
\label{tab:distabletref}
\renewcommand{\arraystretch}{1.8}
\begin{tabular}{c|l}
Filtration Level & Disrespectful Polynomial, $\mathbf{dis}_{U_i}^S(3_1)$  \\ \hline
0            
&
$1t^3$                               \\ \hline
1          & $9t^1$ 
                  \\ \hline
2          & $10t^0$

\end{tabular}
\renewcommand{\arraystretch}{1}
\end{table}

Finally, Figure \ref{fig:quivers} shows the graphical quivers $\mathcal{Q}_{U_i}^S(3_1)$ for $i\in \{0,1,2\}$, with black edges representing respectful endomorphisms and red edges representing disrespectful endomorphisms at each vertex.

\begin{figure}[ht]
    \centering
    \begin{subfigure}[b]{.39\textwidth}
        \centering
        \includegraphics[width=\textwidth]{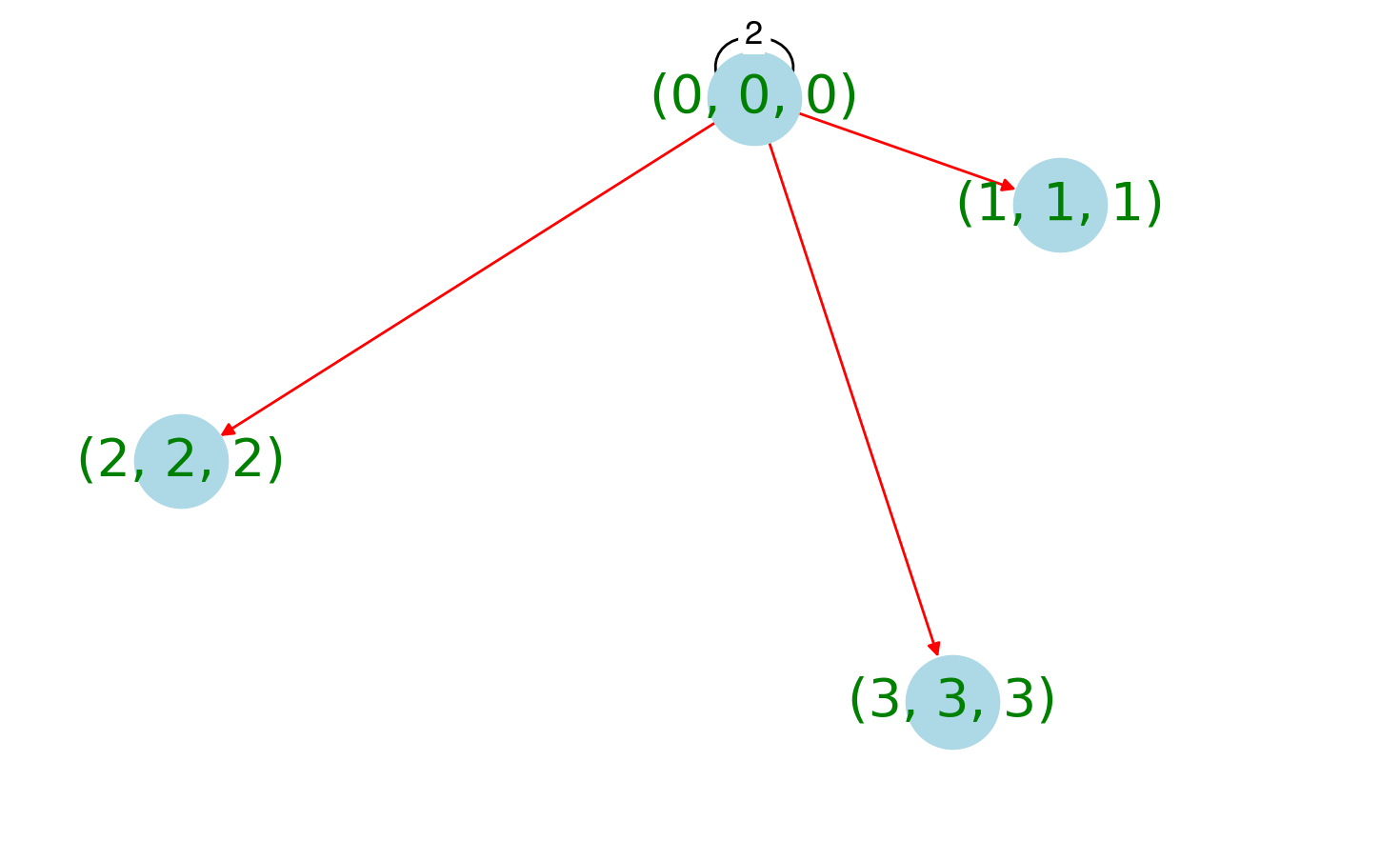}
        \caption{Subquiver $\mathcal{Q}_{U_0}^S(3_1)$}
        \label{fig:310}
    \end{subfigure}
    \hfill
    \begin{subfigure}[b]{.39\textwidth}
        \centering
        \includegraphics[width=\textwidth]{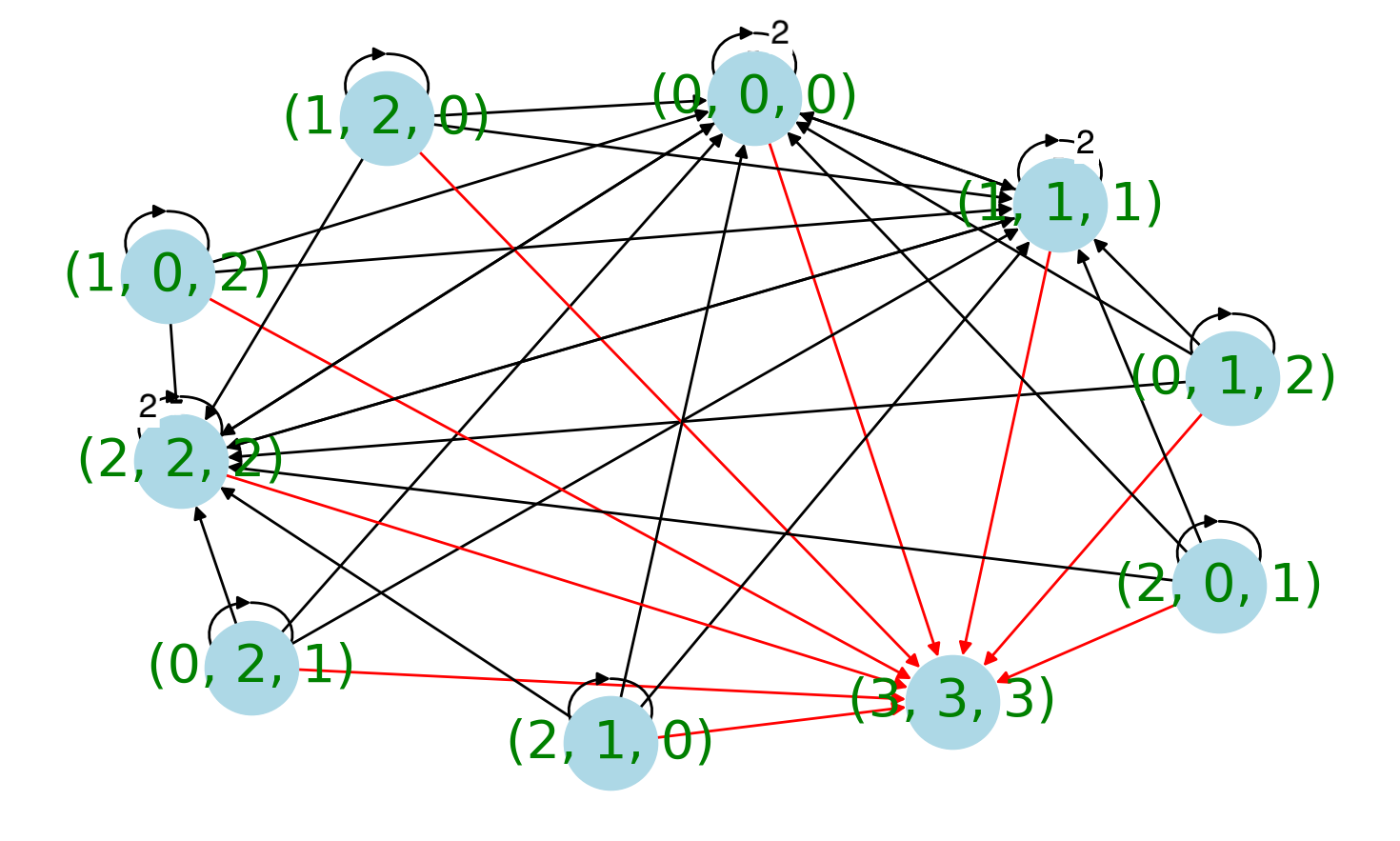}
        \caption{Subquiver $\mathcal{Q}_{U_1}^S(3_1)$}
        \label{fig:311}
    \end{subfigure}
    \hfill
    \begin{subfigure}[b]{.39\textwidth}
        \centering
        \includegraphics[width=\textwidth]{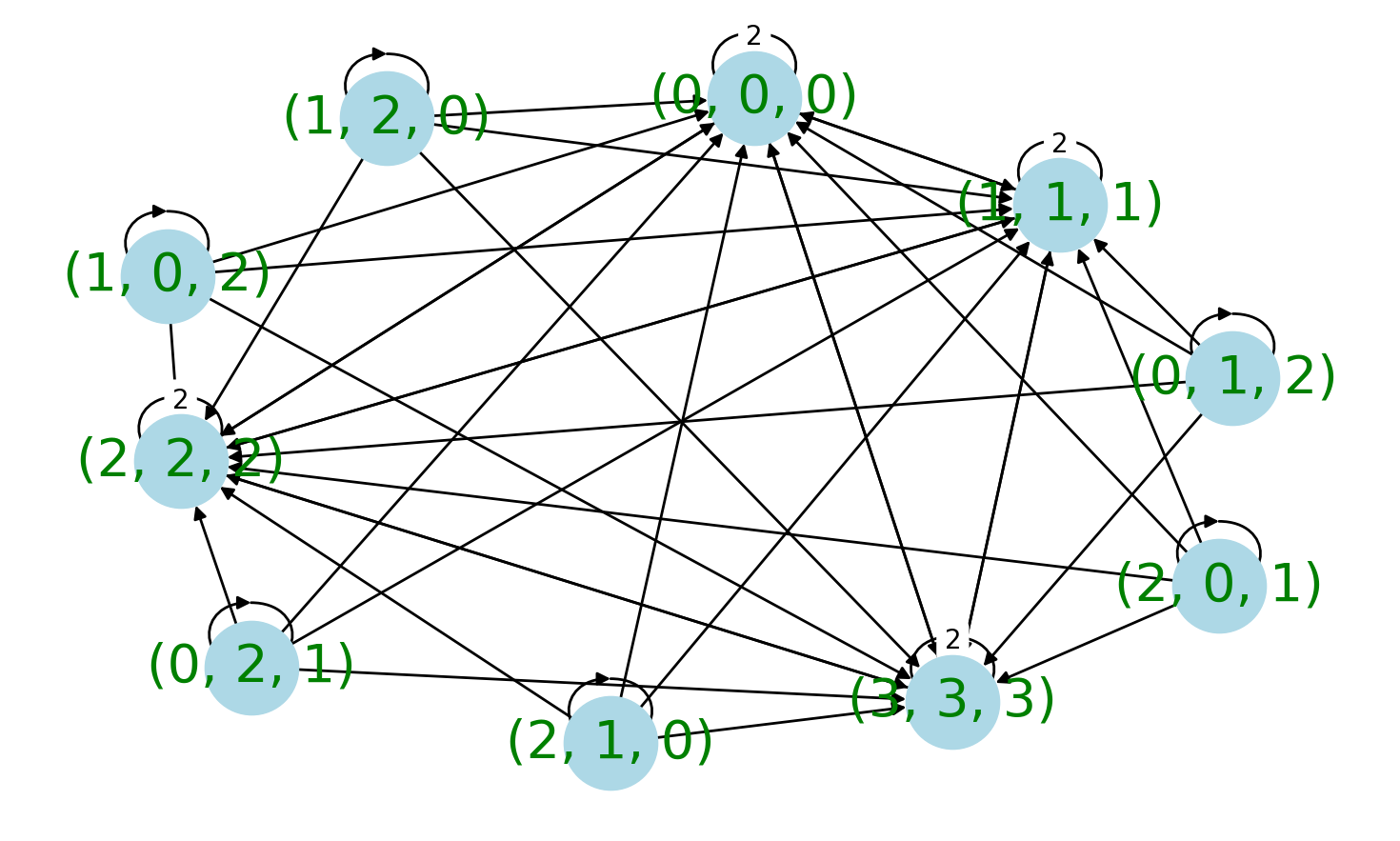}
        \caption{Subquiver $\mathcal{Q}_{U_2}^S(3_1)$}
        \label{fig:312}
    \end{subfigure}
    \caption{Quivers of $3_1$ with respect to the subquandle filtration $U_*$}
    \label{fig:quivers}
\end{figure}

\end{example}

\section{Examples}\label{Examples}

In this section, we provide computations of the various link invariants introduced in this article. We specifically compare our invariants to other well-known link invariants to illustrate their finer resolution in distinguishing links. Furthermore, we include an example demonstrating that the graded disrespectful polynomial is an enhancement of the disrespectful polynomial. Diagrams of the knots and links used in this section can be found in \cite{KA}.

\begin{example}
   Let $(X,\triangleright)$ be the following finite quandle with operation defined by the following table:
        \[
    \begin{array}{c|cccccccccc}
        \triangleright & 0 & 1 & 2 & 3 & 4 & 5 & 6 & 7 & 8 & 9\\
        \hline
        0 & 0 & 9 & 0 & 6 & 7 & 8 & 9 & 8 & 7 & 6 \\
        1 & 5 & 1 & 3 & 5 & 2 & 4 & 3 & 1 & 4 & 2 \\
        2 & 2 & 5 & 2 & 4 & 1 & 3 & 5 & 3 & 1 & 4 \\
        3 & 4 & 4 & 1 & 3 & 5 & 2 & 2 & 5 & 3 & 1 \\
        4 & 1 & 3 & 5 & 2 & 4 & 1 & 4 & 2 & 5 & 3 \\
        5 & 3 & 2 & 4 & 1 & 3 & 5 & 1 & 4 & 2 & 5 \\
        6 & 7 & 8 & 9 & 0 & 6 & 7 & 6 & 0 & 9 & 8 \\
        7 & 9 & 7 & 8 & 9 & 0 & 6 & 8 & 7 & 6 & 0 \\
        8 & 6 & 6 & 7 & 8 & 9 & 0 & 0 & 9 & 8 & 7 \\
        9 & 8 & 0 & 6 & 7 & 8 & 9 & 7 & 6 & 0 & 9 \\
    \end{array}.
    \]

    Let $U_* = (U_i)_{i = 0, 1, 2}$ be a subquandle filtration of $X$, where 
    \begin{align*}
    U_0 &= \{0\},  \\ U_1 &= U_0 \cup \{6, 7, 8, 9\}, \\ U_2 &= U_1 \cup \{1, 2, 3, 4, 5\}.
    \end{align*}
    For a choice of orientation on the knots $5_1$ and $6_1$ we have $\#\operatorname{Col}_X(5_1)  = 30= \#\operatorname{Col}_X(6_1)$. However, when computing the sequence of coloring invariants, we find the following:
    \begin{align*}
        \#\operatorname{Col}_{U_\ast}(5_1) &= (1, 25, 30), \\
        \#\operatorname{Col}_{U_\ast}(6_1) &= (1, 5, 30).
    \end{align*}
    Because $\#\operatorname{Col}_{U_\ast}(5_1) \neq \#\operatorname{Col}_{U_\ast}(6_1)$, we are able to successfully distinguish between the $5_1$ and $6_1$ knots with the sequence of counting invariants.
\end{example}

\begin{example}
    Let $(X,\triangleright)$ be the following finite quandle with operation defined by the following table:
    $$
    \begin{array}{c|cccccccccccc}
    \triangleright & 0 & 1 & 2 & 3 & 4 & 5 & 6 & 7 & 8 & 9 & 10 & 11 \\
    \hline
    0 & 0 & 0 & 0 & 0 & 0 & 0 & 0 & 0 & 0 & 0 & 0 & 0 \\
    1 & 1 & 1 & 1 & 10 & 5 & 6 & 10 & 6 & 5 & 10 & 5 & 6 \\
    2 & 2 & 2 & 2 & 7 & 9 & 4 & 7 & 4 & 9 & 7 & 9 & 4 \\
    3 & 3 & 11 & 8 & 3 & 8 & 11 & 11 & 8 & 3 & 8 & 11 & 3 \\
    4 & 4 & 9 & 7 & 9 & 4 & 7 & 9 & 4 & 7 & 2 & 2 & 2 \\
    5 & 5 & 10 & 6 & 6 & 10 & 5 & 1 & 1 & 1 & 5 & 6 & 10 \\
    6 & 6 & 5 & 10 & 5 & 6 & 10 & 6 & 5 & 10 & 1 & 1 & 1 \\
    7 & 7 & 4 & 9 & 2 & 2 & 2 & 9 & 7 & 4 & 4 & 7 & 9 \\
    8 & 8 & 3 & 11 & 8 & 11 & 3 & 3 & 11 & 8 & 11 & 3 & 8 \\
    9 & 9 & 7 & 4 & 4 & 7 & 9 & 2 & 2 & 2 & 9 & 4 & 7 \\
    10 & 10 & 6 & 5 & 1 & 1 & 1 & 5 & 10 & 6 & 6 & 10 & 5 \\
    11 & 11 & 8 & 3 & 11 & 3 & 8 & 8 & 3 & 11 & 3 & 8 & 11 \\
    \end{array}.
    $$

    Let $U_* = \{U_i\}_{i = 0, 1, 2, 3}$ be a subquandle filtration of $X$, where 
    \begin{align*}
    U_0 &= \{0\},  \\ U_1 &= U_0 \cup \{2, 4, 7, 9\}, \\ U_2 &= U_1 \cup \{3, 8, 11\}, \\
    U_3 &= U_2 \cup \{1, 5, 6, 10\}.
    \end{align*}

    Let $S \subseteq \text{End}(X)$ such that
    \[
    S = \left\{ 
    \begin{array}{cc}
    (11, 0, 0, 0, 0, 0, 0, 0, 0, 0, 0, 0), & (6, 6, 4, 6, 4, 6, 6, 4, 6, 4, 6, 6), \\ (10, 7, 0, 7, 0, 7, 7, 0, 7, 0, 7, 7), &
    (0, 0, 3, 11, 3, 0, 0, 3, 11, 3, 0, 11)
    \end{array}\right\}.
    \]
    For a choice of orientation on $L7a5$ and $L7n1$ we have 
    \[\#\operatorname{Col}_X(L7a5)  = 72= \#\operatorname{Col}_X(L7n1).\] Additionally, the coloring sequences with respect to the subquandle filtration fail to distinguish the two links as 
    \[\#\operatorname{Col}_{U_\ast}(L7a5) = (1, 25, 40, 72) =\#\operatorname{Col}_{U_\ast}(L7n1).\] 
    When we compute the sequence of disrespectful polynomials, we obtain the following: 
    \begin{align*}
        \mathbf{dis}^S_{U_\ast}(L7a5)(t) &= (1t^{3}, 1t^{3} + 8t^{4} + 16t^{1}, 15t^{2} + 16t^{0} + 9t^{1}, 72t^{0}), \\
        \mathbf{dis}^S_{U_\ast}(L7n1)(t) &= (1t^{3}, 1t^{3} + 20t^{4} + 4t^{1}, 27t^{2} + 4t^{0} + 9t^{1}, 72t^{0}).
    \end{align*}
    Because $\mathbf{dis}^S_{U_1}(L7a5)(t) \neq \mathbf{dis}^S_{U_1}(L7n1)(t)$ and $\mathbf{dis}^S_{U_2}(L7a5)(t) \neq \mathbf{dis}^S_{U_2}(L7n1)(t)$, the sequences of disrespectful polynomials are different, and thus we are able to distinguish between $L7a5$ and $L7n1$ by their sequence of disrespectful polynomials.
\end{example}

\begin{example}
    Let $(X,\triangleright)$ be the following finite quandle with operation defined by the following table:
    \[
    \begin{array}{c|cccccc}
    \triangleright & 0 & 1 & 2 & 3 & 4 & 5\\
    \hline
    0 & 0 & 2 & 1 & 0 & 2 & 1 \\
    1 & 2 & 1 & 0 & 1 & 0 & 2 \\
    2 & 1 & 0 & 2 & 2 & 1 & 0 \\
    3 & 3 & 3 & 3 & 3 & 3 & 3 \\
    4 & 5 & 5 & 5 & 4 & 4 & 4 \\
    5 & 4 & 4 & 4 & 5 & 5 & 5 \\

    \end{array}.
    \]
    Let $S \subseteq \text{End}(X)$ be a subset of five endomorphisms, where
    \[
    S = \left\{
    \begin{array}{cc}
(4, 4, 4, 5, 5, 5), (3, 3, 3, 4, 5, 5), (3, 3, 3, 3, 3, 3), (3, 3, 3, 3, 1, 1), (4, 4, 4, 3, 5, 5)

    \end{array}
    \right\}.
    \]
    Let $U_* = \{U\}_{i = 0,1,2}$ be a subquandle filtration of $X$, where
    \begin{align*}
        U_0 &= \{3\} \\
        U_1 &= U_0 \cup \{4,5\} \\
        U_2 &= U_1 \cup \{0,1,2\}.
    \end{align*}
    For a choice of orientation on $L5a1$ and $L7n2$ we obtain 
    $$\#\operatorname{Col}_{U_*}(L5a1) = (1, 9, 30) = \#\operatorname{Col}_{U_*} ( L7n2),$$ and and the also have the same sequence of disrepectful polynomials 
    $$
    \mathbf{dis}_{U_*}^S(L5a1) = (t^{2}, 
t^{0} + 8t^{1},
30t^0) = \mathbf{dis}^S_{U_*}(L7n2).$$

    Thus, the subquandle coloring sequences and the disrespectful polynomial sequences both fail to distinguish these two links.

    However, the sequence of graded disrespectful polynomials are distinct and the polynomials for each link are collected in Tables~\ref{tab:GdistablelinkL5a1} and \ref{tab:GdistablelinkL7n2}.

\begin{table}[ht]
\centering 
\caption{Sequence of graded disrespectful polynomials for $L5a1$.}
\label{tab:GdistablelinkL5a1}
\renewcommand{\arraystretch}{1.8}
\begin{tabular}{c|l}
Filtration Level & Graded Disrespectful Polynomial, $G\mathbf{dis}_{U_i}^S(L5a1)$ \\ \hline
0           & 
$1t_{1,1}^{2}t_{0,1}^{3}$                               \\ \hline
1          & $1t_{1,1}^{2}t_{0,1}^{3} + 4t_{1,1}^{2}t_{1,2}^{1}t_{0,1}^{1}t_{2,1}^{1} + 4t_{1,1}^{3}t_{0,1}^{1}t_{2,1}^{1}$ 
                  \\ \hline
2          & $4t_{1,1}^{2}t_{0,1}^{3} + \mathbf{6t_{1,2}^{1}t_{1,1}^{2}t_{0,1}^{2} + 12t_{1,2}^{2}t_{1,1}^{1}t_{0,1}^{1}t_{2,1}^{1}} + 4t_{1,1}^{2}t_{1,2}^{1}t_{0,1}^{1}t_{2,1}^{1} + 4t_{1,1}^{3}t_{0,1}^{1}t_{2,1}^{1}$

\end{tabular}
\renewcommand{\arraystretch}{1}
\end{table}

\begin{table}[ht]
\centering 
\caption{Sequence of graded disrespectful polynomials for $L7n2$.}
\label{tab:GdistablelinkL7n2}
\renewcommand{\arraystretch}{1.8}
\begin{tabular}{c|l}
Filtration Level & Graded Disrespectful Polynomial, $G\mathbf{dis}_{U_i}^S(L7n2)$ \\ \hline
0            
&
$1t_{1,1}^{2}t_{0,1}^{3}$                               \\ \hline
1          & $1t_{1,1}^{2}t_{0,1}^{3} + 4t_{1,1}^{2}t_{1,2}^{1}t_{0,1}^{1}t_{2,1}^{1} + 4t_{1,1}^{3}t_{0,1}^{1}t_{2,1}^{1}$ 
                  \\ \hline
2          & $4t_{1,1}^{2}t_{0,1}^{3} + \mathbf{6t_{1,2}^{2}t_{1,1}^{1}t_{0,1}^{1}t_{2,1}^{1} + 12t_{1,2}^{1}t_{1,1}^{2}t_{0,1}^{2}} + 4t_{1,1}^{2}t_{1,2}^{1}t_{0,1}^{1}t_{2,1}^{1} + 4t_{1,1}^{3}t_{0,1}^{1}t_{2,1}^{1}$

\end{tabular}
\renewcommand{\arraystretch}{1}
\end{table}

Therefore, since $G\mathbf{dis}_{U_2}^S(L5a1) \neq G\mathbf{dis}_{U_2}^S(L7n2)$, we are able to distinguish between $L5a1$ and $L7n2$ using the sequence of graded disrespectful polynomials.
\end{example}

\begin{example} Let $(X,\triangleright)$ be the following finite quandle with operation defined by the following table:
\[
    \begin{array}{c|cccccccc}
    \triangleright & 0 & 1 & 2 & 3 & 4 & 5 & 6 & 7 \\
    \hline
    0 & 0 & 6 & 6 & 6 & 6 & 0 & 0 & 0 \\
    1 & 4 & 1 & 4 & 2 & 3 & 3 & 2 & 1 \\
    2 & 3 & 3 & 2 & 4 & 1 & 4 & 1 & 2 \\
    3 & 2 & 4 & 1 & 3 & 2 & 1 & 4 & 3 \\
    4 & 1 & 2 & 3 & 1 & 4 & 2 & 3 & 4 \\
    5 & 5 & 0 & 0 & 0 & 0 & 5 & 5 & 5 \\
    6 & 6 & 5 & 5 & 5 & 5 & 6 & 6 & 6 \\
    7 & 7 & 7 & 7 & 7 & 7 & 7 & 7 & 7 \\
    \end{array}.
    \]
Let $S \subseteq \text{End}(X)$ where $S = \{(5, 7, 7, 7, 7, 5, 5, 6)\}.$
Let $U_* = \{U\}_{i = 0, 1, 2, 3}$ be a subquandle filtration of $X$ where 
\begin{align*}
        U_0 &= \{1\} \\
        U_1 &= U_0 \cup \{2, 3, 4\} \\
        U_2 &= U_1 \cup \{0, 5, 6\} \\
        U_3 &= U_2 \cup \{7\}.
    \end{align*}
For a choice of orientation on $L7a2$ and $L7a5$ we have
$$\#\operatorname{Col}_X(L7a5)  = 40= \#\operatorname{Col}_X(L7a2).$$ 
Additionally, the in-degree polynomial fails to distinguish these two links as 
$$
\Phi_X^{\textup{deg}^+,\phi}(L7a5)  = u^{16} + u^9 + 2u^4 + 2u^3 + u + 33 = \Phi_X^{\textup{deg}^+,\phi}(L7a2).
$$
When we compute the sequence of disrespectful polynomials we obtain the following:
    \begin{align*}
        \mathbf{dis}^S_{U_\ast}(L7a5)(t) &= (1t^{1}, 16t^{1}, 9t^{0} + 16t^{1}, 40t^{0}), \\
        \mathbf{dis}^S_{U_\ast}(L7a2)(t) &= (1t^{1}, 4t^{1}, 9t^{0} + 4t^{1}, 40t^{0}).
    \end{align*}
Hence, while the basic coloring invariants and the in-degree polynomials are unable to distinguish these two links, we are able to distinguish $L7a5$ and $L7a2$ using their sequences of disrespectful polynomials.

\end{example}

\section{Acknowlegement}\label{Acknowlegement}
The authors express their gratitude to Hamilton College for supporting this research through funding provided by the Wood Family Gift in STEM. Additionally, the authors acknowledge the use of Gemini 3.1 Pro for language refinement, grammar editing, and improving textual flow. All mathematical results, proofs, and original concepts were developed independently by the authors, who take full responsibility for the final manuscript.

\bibliographystyle{plain}
\bibliography{Ref}

\end{document}